\documentclass[10pt]{amsart}

\usepackage{hyperref}
\usepackage{xcolor}
\usepackage{amssymb}
\usepackage{amsmath}

\let\newpf\proof \let\proof\relax

\def\bm{\begin{matrix}}
\def\em{\end{matrix}}

\newcommand{\bt}{\begin{thm}}
\newcommand{\et}{\end{thm}}

\newcommand{\bl}{\begin{lemma}}
\newcommand{\el}{\end{lemma}}

\newcommand{\beq}{\begin{eqnarray}}
\newcommand{\eeq}{\end{eqnarray}}

\def\be{\begin{equation}}
\def\ee{\end{equation}}

\def\ba{{\begin{align}}}
\def\ea{{\end{align}}}

\def\0{{\mathbf 0}}

\newtheorem{thm}{Theorem}[section]

\newtheorem{lemma}[thm]{Lemma}

\theoremstyle{remark}
\newtheorem{rem}{Remark}[section]

\numberwithin{equation}{section}

\def \bn {\hfill \\ \smallskip\noindent}

\theoremstyle{definition}

\def\proof{\bn {\bf Proof.} }

\def\note#1
{\marginpar

{\tiny $\leftarrow$
\par
\hfuzz=20pt \hbadness=9000 \hyphenpenalty=-100 \exhyphenpenalty=-100
\pretolerance=-1 \tolerance=9999 \doublehyphendemerits=-100000
\finalhyphendemerits=-100000 \baselineskip=6pt
#1}\hfuzz=1pt}

\newcommand{\M}{{\mathbb M}}

\newcommand{\Q}{{\mathbb Q}}
\newcommand{\R}{{\mathbb R}}
\newcommand{\T}{{\mathbb T}}

\newcommand{\Z}{{\mathbb Z}}

\def\B0{{\bold{0}}}

\catcode`\@=12

\def\Empty{}
\newcommand\oplabel[1]{
  \def\OpArg{#1} \ifx \OpArg\Empty {} \else
  	\label{#1}
  \fi}

\newcommand{\comm}[1]{}
\newcommand{\comment}[1]{}

\begin{document}

\title{Singular continuous spectrum for singular potentials}

\author{SVETLANA JITOMIRSKAYA  AND FAN YANG}




\begin{abstract}

We prove that Schr\"odinger operators with meromorphic potentials
$(H_{\alpha,\theta}u)_n=u_{n+1}+u_{n-1}+
\frac{g(\theta+n\alpha)}{f(\theta+n\alpha)} u_n$ have purely singular
continuous spectrum on the set $\{E: L(E)<\delta{(\alpha,\theta)}\}$,
where $\delta$ is an explicit function, and $L$ is the Lyapunov
exponent. This extends results of \cite{maryland} for the Maryland
model and of \cite{ayz1} for the almost Mathieu operator, to the general family of meromorphic potentials.

\end{abstract}

\maketitle

\section{Introduction}
We study operators of the form:
\begin{equation}\label{analyticmodel}
(H_{\alpha,\theta}u)_n=u_{n+1}+u_{n-1}+\frac{g(\theta+n\alpha)}{f(\theta+n\alpha)} u_n
\end{equation}
acting on $l^2(Z)$, where $\alpha \in \R\setminus \Q$ is the
frequency, $\T=\R/\Q$, $\theta \in \T$ is the phase,  $f$ is an
analytic function and $g$ is Lipshitz.
This class contains all meromorphic potentials and therefore both the almost
Mathieu 
($f\equiv 1$, $g=\lambda \cos 2\pi\theta$) and Maryland ($f=\sin 2\pi\theta$, $g=\lambda \cos
2\pi\theta$) families as particular cases. 

Let $\frac{p_n}{q_n}$ be the continued fraction approximants of $\alpha\in \R \setminus \Q$. 

Assume $g/f$ has $m$  poles, $m\ge 0$. We denote them by $\theta_i,
\;i=1,...,m,$ including multiplicities. We now define index $\delta$ as follows:

 \begin{equation} \label{delta}
  \delta( \alpha,\theta) = \limsup_{n\rightarrow \infty} \dfrac{\sum_{i=1}^m \ln\|q_n(\theta-\theta_i) \|_{\R / \Z}+ \ln q_{n+1}} {q_n}.
 \end{equation}   
where 
$\|x\|_{\R / \Z}=min_{l\in \Z}|x-l|$. Let $L(E)$ be the Lyapunov
exponent, see (\ref{L}). $L$ depends also on $\alpha$ but we suppress
it from the notation as we keep $\alpha$ fixed.
 
 Our main result is:
 \begin{thm}\label{sc}Let  $ \delta(\alpha,\theta) $ be as in (\ref{delta}). Then
 \begin{enumerate}
 \item $ H_{\alpha, \theta} $ 
has no eigenvalues on $\{ E:L(E)< \delta(\alpha, \theta)\}$. 
 \item If $L(E)>0$ for a.e. $E$ (in particular, if $m>0$), then $ H_{\alpha, \theta} $ has purely singular continuous spectrum on
$\{ E:L(E)< \delta(\alpha, \theta)\}$.
 \end{enumerate}
\end{thm}
 {\bf Remark.} Since absence of absolutely continuous spectrum follows
 from a.e.positivity of the Lyapunov exponents and holds
 for all unbounded potentials \cite{simonspencer}, part (2) immediately follows
 from part (1), on which we therefore concentrate.

Recently, there has been an increased interest in obtaining arithmetic
conditions (in contrast to a.e. statements) for various quasiperiodic spectral
results. In particular, there have been remarkable advances in the
theory of the almost Mathieu operator

\begin{equation}\label{amo}
(H_{\lambda,\alpha,\theta}u)_n=u_{n+1}+u_{n-1}+\lambda \cos 2\pi(\theta+n\alpha)u_n
\end{equation}
(see e.g. \cite{lastrev,jmrev} for the review and background in
physics). 
Define
\begin{equation}
\beta=\beta(\alpha)=\limsup_{n\rightarrow\infty}\frac{\ln q_{n+1}}{q_n},
\end{equation} which describes how Liouvillian  $\alpha$ is. We say
that $\alpha$ is Diophantine if $\beta(\alpha)=0.$ Note that for almost
 every phase $\theta$ (only depends on $\alpha$) we have $\delta(\alpha,\theta)=\beta(\alpha).$



It was conjectured  in 1994 \cite{1994} that
$\lambda=e^\beta$ is the phase transition point from singular
continuous spectrum to pure point spectrum for $\alpha$-Diophantine
$\theta$ (and that the transition is at larger $\lambda$ for non-$\alpha$-Diophantine
$\theta$). The history of partial results towards this conjecture
include \cite{j,AJ1}.
Recently, Avila, You and Zhou  proved \cite{ayz1}
\begin{thm}\label{lanaconjecture} 
For $\lambda>e^\beta$, the spectrum is pure point with exponentially decaying eigenfunctions for a.e. $\theta,$ and for $1<\lambda<e^\beta$, the spectrum is purely singular continuous for all $\theta$.
\end{thm}
{\bf Remark.} The spectrum is known to be absolutely continuous for
all $\alpha,\theta$ for $\lambda<1$ (the final result in \cite{Avila_2008}).

A fully arithmetic version of the localization statement
\begin{thm}\label{liu1} 
For $\lambda>e^\beta$, the spectrum is pure point with exponential decaying eigenfunctions for $\alpha$-Diophantine $\theta$ 
\end{thm}
was established recently in \cite{jl1}.

Define also
\begin{equation}\label{G.delta}
\gamma=\gamma(\alpha,\theta)=\limsup_{n \to \infty} \dfrac{-\ln \vert\vert 2 \theta+n \alpha \vert\vert_{\R/\Z} }{\vert n \vert} ~\mbox{.}
\end{equation} 
We say that  $\theta$ is $\alpha$-Diophantine if $\gamma(\alpha,\theta)=0.$

It was also conjectured  in \cite{1994,Jitomirskaya_review_2007} that
$\lambda=e^\gamma$ is the phase transition point from singular
continuous spectrum to pure point spectrum for Diophantine
$\alpha$ (and that the transition is at larger $\lambda$ for non-Diophantine
$\alpha$). Partial results towards this conjecture include \cite{js,j}.
The conjecture was recently fully established in \cite{jl2}:
\begin{thm}\label{liu2} 
For $\lambda>e^{\gamma(\alpha,\theta)}$, the spectrum is pure point with exponentially
decaying eigenfunctions for Diophantine $\alpha,$ and  for
$\lambda<e^{\gamma(\alpha,\theta)}$, the spectrum is singular
continuous for all $\alpha.$
\end{thm}
Therefore, for the almost Mathieu operator the precise transition from
pure point to singular continuous spectrum is understood for 
either Diophantine $\alpha$ and all $\theta$ or for all $\alpha$ and
$\alpha$-Diophantine $\theta,$ but not yet for all parameters.

Another case with a significant recent  arithmetic results is the
Maryland model
\begin{equation}\label{maryland}
 (H_{\lambda,\alpha,\theta}u)_n=u_{n+1}+u_{n-1}+\lambda \tan \pi(\theta+n\alpha)u_n
\end{equation}
It is the prototypical operator of form (\ref{analyticmodel}).
This model was proposed by Grempel, Fisherman, and Prange
\cite{grempel1982localization} as a linear version of the quantum
kicked rotor. 
 It
is an exactly solvable example of the family of incommensurate models,
thus attracting continuing interest in physics,
e.g. 
\cite{GKDS2014}.  The complete
description of spectral transitions for the Maryland model (depending
on arithmetic properties of all parameters)
was given recently in  \cite{maryland}.

Namely, an index $\delta(\alpha,\theta) \in [-\infty,\infty]$ was  introduced
in \cite{maryland}:
 \begin{equation}\label{marylanddelta}
 \delta(\alpha,\theta)=\limsup_{n \rightarrow \infty} \frac{\ln \|q_n(\theta-\frac{1}{2})\|_{\R / \Z}+\ln q_{n+1}}{q_n}
 \end{equation}

 The main result of \cite{maryland} regarding the singular continuous part is:
 
 \begin{thm}\label{jlsingular}\cite{maryland}
 $H_{\alpha,\theta}$  has purely singular continuous spectrum on 
$\{E: L(E)<\delta(\alpha,\theta) \}$. 
 \end{thm}
It is sharp since
\begin{thm}\label{jlpoint}\cite{maryland}
 $H_{\alpha,\theta}$  has pure point spectrum on 
$\{E: L(E) > \delta(\alpha,\theta) \}$. 
 \end{thm}

 Our result therefore is an extension of Theorem \ref{jlsingular} (to
 which Theorem \ref{sc} specializes for $f=\cos 2\pi\theta$, $g=\lambda \sin
2\pi\theta$)
 to the general family of singular potentials. For $f\equiv 1, g=\lambda \cos 2\pi\theta$ we recover
 the singular continuous part of Theorem \ref{lanaconjecture} (note
 that the
 proof of \cite{ayz1} also extends in this case to  $f\equiv1$ and a
 Lipshitz condition on $g$ without many changes). Theorem \ref{jlpoint} shows that our
result is sharp for the Maryland model. However, Theorems \ref{liu1}, \ref{liu2}
 show that it is not sharp for the almost Mathieu operator other than
 for $\alpha$-Diophantine $\theta.$ Based on this,
 we do not expect sharpness for general Lipshitz or even analytic potentials ($f\equiv
 1$),  and
 conjecture that sharpness (that is point spectrum in the
 complementary regime other than possibly on the transition line) may
 be a corollary of certain monotonicity.

\section{Preliminaries: cocycle, Lyapunov exponent}  
Assume without loss of generality,
$ f(\theta)=(e^{2\pi i\theta}-
e^{2\pi i\theta_1}) \cdots (e^{2\pi i\theta}-
e^{2\pi i\theta_m})$, $m=1,\cdots$. 

Let $\Theta=\cup_{l=1}^m {\theta_l+\Z\alpha+\Z}$.
From now on we fix $E$ in the spectrum and $\theta\in \Theta^c$ such that $L(E)<\delta(\alpha, \theta)$. 
We will show $H_{\lambda, \alpha, \theta}$ cannot have an eigenvalue
at $E.$

A formal solution of the equation  $ H_{\alpha, \theta}u=Eu $ can be reconstructed via the following equation
 \begin{align*}
\left (\begin{matrix}  u_{n+1} \\ u_{n} \end{matrix} \right )= A(\theta+n\alpha)\left (\begin{matrix}  u_{n} \\  u_{n-1}  \end{matrix} \right )
 \end{align*} 
 where
$
A(\theta)=\left (\begin{matrix}  E-\frac{g(\theta)}{f(\theta)}   &  -1 \\ 1  &   0 \end{matrix} \right )
$ is the so-called transfer matrix.

 The pair $(\alpha,A)$ is the cocycle corresponding to the operator  (\ref{analyticmodel}). 
 It can be viewed as a linear skew-product $(x,\omega)\mapsto(x+\alpha,A(x)\cdot \omega)$.
 Generally, one can define $M_n$ for an invertible cocycle $(\alpha, \M)$ by $(\alpha,M)^n=(n\alpha,M_n)$, $n \in Z$ so that for $n \geq 0$:
 $$M_n(x)=M(x+(n-1)\alpha)M(x+(n-2)\alpha) \cdots M(x),$$ and $M_{-n}(x)=M_n(x-n\alpha)$.
 
The Lyapunov exponent of a cocycle $(\alpha,M)$ is defined by$$L(\alpha,M)=\lim_{n \rightarrow \infty}\frac{1}{n} \int_{\T} \ln \|M_n(x)\|\mathrm{d}x.$$

Let $A(x)=\frac{1}{f(x)}D(x)$ where 
\begin{align*}
D(x)=
\left(
\begin{matrix}
    E f(x)- g(x)        &  -f(x)
 \\       f(x)              &   0
 \end{matrix}
\right)
 \end{align*} 
be the regular part of $A(x)$.
Since $\int_{\T} \ln{|f(x)|} \mathrm{d}x=0$, we have 
\begin{equation}\label{L}
L(E):=L({\alpha, A} )= L(\alpha, D).
\end{equation}

  \begin{lemma} \label{lana}\cite{AJ1}
 Let $\alpha\in \R\backslash\Q $,\ $\theta\in\R$ and $0\leq j_0 \leq q_{n}-1$ be such that 
 $$\mid \sin \pi(\theta+j_{0}\alpha)\mid = \inf_{0\leq j \leq q_{n}-1} \mid \sin \pi(\theta+j\alpha)\mid ,$$
 then for some absolute constant $C>0$,
 $$-C\ln q_{n} \leq \sum_{j=0,j\neq j_0}^{q_{n}-1} \ln \mid \sin \pi (\theta+j\alpha) \mid+(q_{n}-1)\ln2 \leq C\ln q_n$$
 \end{lemma}
 
 We will also use that the denominators of continued fraction approximants of $\alpha$ satisfy
 
 $$ \| k\alpha \|_{\R \backslash \Z} \geq \| q_n \alpha \|_{\R \backslash \Z}, 1 \leq k < q_{n+1}, $$ 
 and
\begin{equation}\label{contfraction}
 \dfrac{1}{2q_{n+1}}  \leq  \|q_n \alpha\|_{\R \backslash \Z}  \leq  \dfrac{1}{q_{n+1}}.
\end{equation} 

A quick corollary of subadditivity and unique ergodicity is the
following upper semicontinuity statement:
 \begin{lemma}\label{LcontrolA} (e.g. \cite{AJ2})
Suppose $(\alpha,A)$ is a continuous cocycle. Then for any $ \varepsilon>0$, there exists $C(\varepsilon)>0$, such that for any $x\in \T$ we have
$$\|A_n(x) \| \leq C e ^{n(L(A)+\varepsilon)}.$$
\end{lemma}

\begin{rem}\label{partf}
Applying this to 1-dimensional continuous cocycles, we get that if $g$
is a continuous function such that $\ln |g| \in L^1(\T)$,
then $$|\prod_{l=a}^b g (x+l\alpha) | \leq e^{(b-a+1)(\int \ln|g| \mathrm{d}\theta+\varepsilon)}.$$
\end{rem}

\section{Absence of point spectrum}
Let $\varphi$ be a solution to $H_{\alpha, \theta}\varphi =E\varphi$ satisfying 
$\| \left(
\begin{matrix}
\varphi_0\\
\varphi_{-1}
\end{matrix}
\right)
\|=1$.
We have the following restatement of Gordon's lemma. We state a precise form that will be convenient for us.
\begin{thm}\label{absenceofpp}
If there exists a constant $c>0$ and a subsequence $q_{n_i}$ of $q_n$ such that the following estimates holds:
\begin{equation}\label{absenceofpp_1}
\|(A_{q_{n_i}}^2(\theta)-A_{2q_{n_i}}(\theta))
\left(
\begin{matrix}
\varphi_0\\
\varphi_{-1}
\end{matrix}
\right)
\|\leq e^{-cq_{n_i}}
\end{equation}
and
\begin{equation}\label{absenceofpp_2}
\|(A^{-1}_{q_{n_i}}(\theta)
-A^{-1}_{q_{n_i}}(\theta-q_{n_i}\alpha))
\left(
\begin{matrix}
\varphi_0\\
\varphi_{-1}
\end{matrix}
\right)
\|\leq e^{-cq_{n_i}},
\end{equation}
then we have 
\begin{equation}\label{maxineq}
\max\{ \|(\begin{array}{cc}\varphi_{q_{n_i}} \\ \varphi_{q_{n_i}-1}\end{array})\|, 
           \|(\begin{array}{cc}\varphi_{-q_{n_i}} \\ \varphi_{-q_{n_i}-1}\end{array})\|,
           \|(\begin{array}{cc}\varphi_{2q_{n_i}} \\ \varphi_{2q_{n_i}-1}\end{array})\| \} \geq \frac{1}{4}.
\end{equation}
\end{thm}

\begin{proof}
This is a standard argument, going back to \cite{Gordon_1976}. The key idea is to use the following two equalities:
\begin{align*}\label{gordon}
\left\lbrace
\begin{matrix}
A_{q_{n_i}}(\theta)-\mathrm{Tr} A_{q_{n_i}}(\theta) \cdot Id + A_{q_{n_i}}^{-1}(\theta)=0\\
A_{q_{n_i}}^2(\theta)-\mathrm{Tr} A_{q_{n_i}}(\theta) \cdot A_{q_{n_i}}(\theta) +Id=0
\end{matrix}
\right.
\end{align*}
and separate the cases $|\mathrm{Tr}A_{q_{n_i}} (\theta)| > \frac{1}{2}$,$|\mathrm{Tr}A_{q_{n_i}} (\theta)| < \frac{1}{2}.$
 
\end{proof} $\hfill{} \Box $

\subsection{Proof of Theorem \ref{sc}}
Assume $\varphi$ is a decaying solution of $H_{\alpha, \theta}\varphi = E\varphi$, satisfying $\|(\begin{array}{cc} \varphi_0 \\ \varphi_{-1} \end{array})\|=1$. 
On one hand, it must be true that for any $\eta>0$, there exists $N$ such that 
$\|\left(
\begin{matrix}
\varphi_k\\
\varphi_{k-1}
\end{matrix}
\right)
\|\leq \eta$ for $|k|>N$.
On the other hand, we will prove the following lemma in section $5$:
\begin{lemma}\label{A}
For any $ \varepsilon >0$ there exists a subsequence 
$\{q_{n_i}\}$ of $\{q_n\}$ so that we have the following estimates:

\begin{equation}\label{Ainverse}
\|(A^{-1}_{q_{n_i}}(\theta)
-A^{-1}_{q_{n_i}}(\theta-q_{n_i}\alpha))
\left(
\begin{matrix}
\varphi_0\\
\varphi_{-1}
\end{matrix}
\right)
\| \leq e^{q_{n_i}(L(E)-\delta(\alpha, \theta)+4\varepsilon)},
\end{equation}
and
\begin{equation}\label{Asquare}
\|(A_{q_{n_i}}^2(\theta)-A_{2q_{n_i}}(\theta))
\left(
\begin{matrix}
\varphi_0\\
\varphi_{-1}
\end{matrix}
\right)
\| \leq e^{q_{n_i}(L(E)-\delta(\alpha, \theta)+4\varepsilon)}.
\end{equation}

\end{lemma}
Then combining Lemma \ref{A} and Theorem \ref{absenceofpp} we get a contradiction, which shows the absence of point spectrum.\qed

\section{key lemmas}
Let $| \sin \pi (\theta-\theta_l+j_l \alpha) | =\inf_{0\leq j \leq q_n-1} | \sin \pi (\theta-\theta_l+j\alpha)|$.
 
 \begin{lemma} \label{minimal}
 
If $\delta(\alpha,\theta) > 0$, then for any $\varepsilon>0$, there exists a subsequence $q_{n_i}$ of $q_n$ such that the following estimate holds
 $$\prod_{l=1}^m | \sin \pi (\theta-\theta_l+j_l \alpha) | \geq \frac{e^{q_{n_i}(\delta-\frac{\varepsilon}{2})}}{q_{n_i+1}}.$$
\end{lemma} 

 \begin{proof}
  By the definition of $\delta(\alpha, \theta)$, there exists a subsequence $q_{n_i}$ of $q_n $ such that
  $$\dfrac{\sum_{l=1}^m \ln\|q_{n_i}(\theta-\theta_l) \|+ \ln q_{n_i+1}} {q_{n_i}} >\delta(\alpha,\theta)-\frac{\varepsilon}{4},$$
  thus
  $$\|q_{n_i}(\theta-\theta_1)\| \cdots \|q_{n_i}(\theta-\theta_m)\| > \frac{e^{q_{n_i} (\delta-\frac{\varepsilon}{4})}}{q_{n_i+1}}.$$
  In particular,
  $\|q_{n_i} (\theta -\theta_l)\| >  \frac{e^{q_{n_i} (\delta-\frac{\varepsilon}{4})}}{q_{n_i+1}} $ for any $1\leq l \leq m$.
  Since
  \begin{align*}
           &| \sin \pi (\theta-\theta_l+j_l \alpha)|\\
   \geq &2\|(\theta-\theta_l+j_l \alpha)\|\\
   \geq &\frac{2\|q_{n_i}(\theta-\theta_l+j_l \alpha)\|}{q_{n_i}} \\
   \geq &\frac{2\|q_{n_i}(\theta-\theta_l)\|-\frac{2q_{n_i}}{q_{n_i+1}}}{q_{n_i}} \\
   \geq &\frac{\|q_{n_i}(\theta-\theta_l)\|}{q_{n_i}} .
 \end{align*}
  We have
  $$\prod_{l=1}^m | \sin \pi (\theta-\theta_l+j_l \alpha) | 
  \geq \prod_{l=1}^m  \frac{\|q_{n_i}(\theta-\theta_l)\|}{q_{n_i}} 
  > \frac{e^{q_{n_i} (\delta-\frac{\varepsilon}{4})}}{q_{n_i+1}} \cdot \frac{1}{(q_{n_i})^m}
  > \frac{e^{q_{n_i}(\delta-\frac{\varepsilon}{2})}}{q_{n_i+1}}$$
 \end{proof} $\hfill{} \Box $
 
\begin{lemma}\label{f}
 The following estimate holds
 
$$\prod_{j=0}^{q_{n_i}-1}|f(\theta+j\alpha )| \geq \dfrac{e^{q_{n_i}(\delta - \varepsilon)}}{q_{n_i+1}}.$$

\end{lemma}
 
 \begin{proof}
\begin{align*}
\prod_{j=0}^{q_{n_i}-1}|f(\theta+j\alpha )|
&=2^{m q_{n_i}}\prod_{l=0}^m \prod_{j=0}^{q_{n_i}-1}| \sin \pi (\theta-\theta_l+j\alpha)| \\
&=2^{m q_{n_i}}\left( \prod_{l=0}^m \prod_{j=0,j\neq j_l}^{q_{n_i}-1}| \sin \pi (\theta-\theta_l+j\alpha)| \right) \cdot \left( \prod_{l=0}^m |\sin \pi (\theta-\theta_l+j_l\alpha)| \right).
\end{align*}

Combining Lemma $\ref{lana}$ and $\ref{minimal}$,
$$\prod_{j=0}^{q_{n_i}-1}|f(\theta+j\alpha )| \geq 2^{mq_{n_i}} e^{m(-C \ln q_{n_i} -(q_{n_i}-1)\ln 2)} \cdot \frac{e^{q_{n_i}(\delta-\frac{\varepsilon}{2})}}{q_{n_i+1}} \geq \frac{e^{q_{n_i}(\delta-\varepsilon)}}{q_{n_i+1}}.$$

 \end{proof} $\hfill{} \Box $

\section{proof of lemma \ref{A}}
 
We give a detailed proof of (\ref{Ainverse}). (\ref{Asquare}) could be proved in a similar way.
\begin{proof}
Let $$A^{-1}(x)=\frac{1}{f(x)} \left( \begin{array}{cc} 0 & f(x) \\ -f(x)& Ef(x)-g(x)  \end{array} \right) \triangleq \frac{F(x)}{f(x)}. $$
Consider $$\Psi_{n_i}=\left( A_{q_{n_i}}^{-1} (\theta)-A_{q_{n_i}}^{-1}(\theta- q_{n_i} \alpha)\right) \left( \begin{array}{cc}\varphi_0 \\ \varphi_{-1}  \end{array} \right).$$

For simplicity let us introduce some notations: fixing $\theta$, for any function $z(x)$ on $\T$ denote $z_j=z(\theta+j\alpha)$; for any matrix function $M(x)$ denote $M^j=M(\theta+j\alpha)$.
Then, by telescoping,
   \begin{align*}
     \Psi_{n_i}
   =& 
     \left( \frac{F^0}{f_0} \frac{F^1}{f_1} \cdots \frac{F^{q_{n_i}-1}}{f_{q_{n_i}-1}} - \frac{F^{-q_{n_i}}}{f_{-q_{n_i}}} \cdots \frac{F^{-1}}{f_{-1}} \right) 
      \left( \begin{array}{cc} \varphi_0 \\ \varphi_{-1}  \end{array} \right) \\
   =& 
     \sum_{j=0}^{q_{n_i}-1} \left( \frac{F^0 F^1 \cdots F^{j-1}}{f_0 f_1\cdots f_{j-1}} \right)
     \left( \frac{F^j}{f_j} -\frac{F^{-q_{n_i}+j}}{f_{-q_{n_i}+j}}\right)
     \left( \frac{F^{-q_{n_i}+j+1} \cdots F^{-1}}{f_{-q_{n_i}+j+1}\cdots f_{-1}} \right)   \left( \begin{array}{cc} \varphi_0 \\ \varphi_{-1}  \end{array} \right),
     \end{align*} 
  where for $j=0$ the first, and for $j=q_{n_i}-1$ the last, multiple are set to be equal to one.
  
  Thus\begin{align*}
\Psi_{n_i}
   =&
     \sum_{j=0}^{q_{n_i}-1} \left( \prod_{l=0}^{j-1}\frac{F^l}{f_l} \right)
     \left( \frac{F^{j}}{f_{j}} -\frac{F^{-q_{n_i}+j}}{f_{-q_{n_i}+j}}\right)
     \left( \begin{array}{cc}\varphi_{-q_{n_i}+j+1} \\ \varphi_{-q_{n_i}+j} \end{array} \right)\\
   =& 
     \sum_{j=0}^{q_{n_i}-1} \left( \prod_{l=0}^{j-1}\frac{F^l}{f_l} \right)
     \left( \frac{F^{j}f_{-q_{n_i}+j}-F^{-q_{n_i}+j}f_{-q_{n_i}+j} +F^{-q_{n_i}+j}f_{-q_{n_i}+j} - F^{-q_{n_i}+j}f_{j}}{f_{j}f_{-q_{n_i}+j}}\right)
     \left( \begin{array}{cc} \varphi_{-q_{n_i}+j+1} \\ \varphi_{-q_{n_i}+j}  \end{array} \right)  \\
   =& 
     \sum_{j=0}^{q_{n_i}-1} \left( \prod_{l=0}^{j-1}\frac{F^l}{f_l} \right) 
     \left( \frac{F^{j}-F^{-q_{n_i}+j}}{f_{j}} \left( \begin{array}{cc} \varphi_{-q_{n_i}+j+1} \\    \varphi_{-q_{n_i}+j} \end{array} \right) 
    +\frac{f_{-q_{n_i}+j}-f_{j}}{f_{j}}\left( \begin{array}{cc} \varphi_{-q_{n_i}+j} \\ \varphi_{-q_{n_i}+j-1} \end{array} \right) \right).
\end{align*}

Since $\phi$ is decaying solution, there exists a constant $C>0$ such that $$\| \left( \begin{array}{cc}\varphi_{k} \\ \varphi_{k-1} \end{array} \right)\| \leq C.$$

Observe that $sup_{\theta} \|F(\theta+q_{n_i}\alpha)-F(\theta)\|<\frac{C}{q_{n_i+1}}$.
Now we can get, using Lemma \ref{LcontrolA}, Remark \ref{partf} and
Lemma \ref{f}  in the second inequality
\begin{align*}
       &\|\left( A_{q_{n_i}}^{-1} (\theta)-A_{q_{n_i}}^{-1}(\theta- q_{n_i} \alpha)\right) \left( \begin{array}{cc}\varphi_0 \\ \varphi_{-1}  \end{array} \right)\|\\
\leq &C\sum_{j=0}^{q_{n_i}-1}\frac{\|\prod_{l=0}^{j-1}F^l\|}{q_{n_i+1}|\prod_{l=0}^{j}f_l|}\\
=     &C\sum_{j=0}^{q_{n_i}-1}\frac{\|\prod_{l=0}^{j-1}F^l \|  |\prod_{l=j+1}^{q_{n_i}-1}f_l |} {q_{n_i+1}|\prod_{l=0}^{q_{n_i}-1}f_l|}\\
\leq &C\frac{q_{n_i} e^{q_{n_i}(L(E)+\varepsilon)} \cdot e^{q_{n_i} \varepsilon}}{e^{q_{n_i} (\delta-\varepsilon)}}\\
\leq & e^{q_{n_i} (L(E)-\delta+4\varepsilon)}.
\end{align*}
\end{proof} $\hfill{} \Box$

 \section*{Acknowledgement}
 F.Y would like to thank Rui Han for his help, useful discussions and encouragement throughout all the work. We also would like to thank Qi Zhou and Wencai Liu for some suggestions.
 F.Y was supported by CSC of China (no.201406330007) and the NSFC
 (no.11571327) and NSF of Shandong Province (grant
 no.ZR2013AM026). She would like to thank her advisor Daxiong Piao
 (Professor at Ocean University of China) for supporting her partly. This research was partially supported by NSF DMS-1401204.

\bibliographystyle{amsplain}

\end{document}